\documentclass[10pt,a4paper]{article}
\usepackage[top=2.5cm, bottom=2.5cm, left=2.2cm, right=2.2cm]{geometry}

\usepackage{amsmath,amsthm,amssymb}
\usepackage[numbers]{natbib}
\usepackage{microtype}
\usepackage{lmodern}
\usepackage{mathrsfs}
\usepackage{color}
\usepackage{graphicx}
\usepackage{yfonts}
\usepackage{subfigure}
\usepackage{rotating}
\usepackage{makecell}

\usepackage{bm}
\usepackage[final]{showkeys}

\numberwithin{equation}{section}
\allowdisplaybreaks

\newtheorem{theorem}{Theorem}[section]
\newtheorem{remark}{Remark}[section]
\newtheorem{definition}{Definition}[section]

\title{\textbf{Generalized Nonlinear Yule Models}}

\author{$\text{Petr Lansky}^\alpha$, $\text{Federico Polito}^{\beta,*}$ \& $\text{Laura Sacerdote}^\beta$\\
	\footnotesize{${}^\alpha$Department of Mathematics and Statistics, Masaryk University, Brno, Czech Republic}\\
	\footnotesize{${}^\beta$Dipartimento di Matematica \emph{G.~Peano}, Universit\`a degli Studi di Torino, Torino, Italy}\\
	\footnotesize{${}^*$Corresponding author. Email: \texttt{federico.polito@unito.it}, Tel: +39 011 6702937}}

\date{}

\begin{document}

	\maketitle

	\begin{abstract}

		With the aim of considering models related to random graphs growth exhibiting persistent memory,
		we propose a fractional nonlinear modification of the classical Yule model often studied in the context of macroevolution.
		Here the model is analyzed and interpreted in the framework of the development of networks such as the World Wide Web.
		Nonlinearity is introduced by replacing the linear birth process governing the growth of the in-links of each
		specific webpage with a fractional nonlinear birth process with completely general birth rates.
		
		Among the main results we derive the explicit distribution of the number of in-links of a webpage chosen uniformly at random recognizing
		the contribution to the asymptotics and the finite time correction. The mean value of the latter distribution is also calculated
		explicitly in the most general case. Furthermore, in order to show the usefulness of our results, we particularize them in
		the case of specific birth rates giving rise to a saturating behaviour, a property that is often observed in nature.
		The further specialization to the non-fractional case allows us to extend the Yule model
		accounting for a nonlinear growth.
		
		\vspace{.2cm}
		
		
		\vspace{.2cm}
		

	\end{abstract}

	\section{Introduction and background}

		The seminal paper \citep{yule} was the original inspiration for many studies appeared from the second half of the last century till now.
		It  contains implicitly the preferential attachment paradigm, a concept that had an enormous success after the appearance of the
		Barab\'asi--Albert model \citep{MR2091634}
		describing the growth of the World Wide Web (WWW). Preferential attachment models are undoubtedly one of
		the most studied and appreciated classes of network growth models
		\citep{krapivsky2000connectivity, PhysRevLett.85.4633, Bollobas2001, krapivsky2002finiteness, MR2788549, dereich2016distances}.
		The relevant literature is vast and spread in many different fields,
		ranging from mathematics
		to physics, computer science, biology, ecology, and many others. It would seem futile to aspire in giving here a thorough review
		of the literature and in the following we will limit ourselves to cite papers only directly relevant for our work.
		
		The Yule model is a continuous time linear model originally motivated by the study of macroevolutionary dynamics. Later, it was revisited
		to describe a variety of phenomena, including the growth of webpages and links in the WWW. A second class of preferential
		attachment models derives from the Simon model \citep{MR0073085}, a discrete time model originally
		proposed to describe the count of words in a
		text and then used in different fields \citep{MR2788549}.  
		
		Yule, Simon and Barab\'asi--Albert models share asymptotic degree distributions with tails characterized by a power-law behaviour.
		This fact has determined a frequent confusion among them. In a recent paper \citep{pachon2015random} we point out the existing
		relationships and differences
		between these three models making use of random graph theory. In particular, we show that the Yule model can be related
		to the continuous-time limit of a sequence of suitably rescaled Simon models.
		The existence of a well determined relationship between Simon and Yule models increases the interest for the Yule model
		itself that is mathematically more tractable than its discrete-time analogue. 
		Beside the preferential attachment assumption, further hypotheses characterize Yule, Simon and Barab\'asi--Albert models.
		In fact they are all Markovian and share linear growth rates.
		The Markov property of the Yule model is determined by the intrinsic
		exponential nature of the waiting times for the appearance of genera and
		of species for each newly created genus (in the original formulation, see \cite{yule}). In the sequel we will refer to webpages in place of
		genera and to in-links in substitution to species.
		
		The success in the modelling a diversified set of phenomena has increased the interest in these models and has suggested to
		attempt possible improvements. Generalizations of the Yule model have already appeared in the literature
		\citep[e.g.][]{maruvka2013model,reed2002size}.
		In \citet{lansky2014role}
		we show that the introduction of the detachment of in-links leads to a good fit of recent WWW data. 
		We realized, however, that even a generalized Yule model (without detachment)
		in the cases of non Markov and nonlinear rates models, or even only nonlinear rates,
		was not yet studied in the literature, to the best of our knowledge. 
		
		In Section \ref{results} we introduce a generalization of the Yule model removing the Markov hypothesis and assuming nonlinear
		rates for the number of in-links growth. In Section \ref{satura} we exemplify our model in the special case of rates of in-links characterized
		by saturation and in presence of non-Markovian memory for the in-link growth. To pursue this last aim we substitute the
		geometric law (corresponding to a linear birth process) for the growth of the number of in-links by a different class
		of counting processes called fractional nonlinear birth processes. These processes are characterized by a parameter that
		accounts for the length of the memory of the process. They coincide with the classical nonlinear birth processes for a specific
		choice of the characterizing parameter. Notice that we consider processes with general nonlinear birth rates.
		The only assumption we introduce is that the rates are such that explosion of the processes in finite time is not allowed.
				
		Fractional point processes in the recent years have been object of intense study and development starting
		from the simplest ones such as the fractional homogeneous Poisson process or the fractional linear birth process and going further
		till processes defined via fairly complex random time-changes.
		For the sake of clarity and completeness, in the Appendix \ref{appe} we report a brief construction of the fractional
		nonlinear birth processes. Here we limit ourselves to refer to some of the papers present in the literature
		such as \citet{MR2007003,MR2535014,MR2120631,politi2011full,MR2990704,MR2835248,MR3239686,MR3304182,MR2899530}.
		The fractional nonlinear birth process was studied in \citet{MR2730651,MR3021490}.
		Some specific cases and related models are given in \citet{MR2880036} and \citet{MR3336844}.

	\section{The Classical Yule Model}
	
		\label{results}
		We consider the growth of a random network described through a Yule model \citep{pachon2015random}.
		An initial webpage appears with a single in-link at
		time $t=0$. Then, each in-link starts duplicating at a constant rate $\lambda>0$. Hence, the number of in-links for
		each webpage evolves as a homogeneous linear birth process (Yule process). In turn, new webpages
		(each of them with a single in-link) are created at a constant rate $\beta>0$. Similarly
		to in-links belonging to the same webpage, also webpages develop independently as a Yule process of parameter $\beta$.
				
		Let us indicate with $N_\beta(t)$ the process counting the number of webpages and with $N_\lambda(t)$ that counting
		the number of in-links of a given webpage.
		Recall now that the state probability distribution $\mathbb{P} ( N_\lambda(t) = n )$, $n \ge 1$, of
		a Yule process is geometric and reads \cite{bailey}
		\begin{align}
			\label{bai}
			\mathbb{P} ( N_\lambda(t) = n ) = e^{-\lambda t} (1-e^{-\lambda t})^{n-1}, \qquad n \ge 1, \: t \ge 0.
		\end{align}
		The Yule process lies in the class of the so-called \emph{processes with the
		order statistic property} \cite{neuts} like for instance the homogeneous Poisson process.
		For the latter, this means that, given the number of occurred event at time $t$, the random instants
		of events occurance are distributed as the order statistics of iid uniform random variables on the interval
		$(0,t)$.
		In general, for a counting process $K(t)$, $t \ge 0$, with the order statistics property, given the number of event
		registered at time $t$, the random jump instants are distributed as the order statistics
		of iid random variables with distribution function given by (\citet{crump}, Theorem 1; see also \citet{feigin,puri})
		\begin{align*}
			k(\tau) = \frac{\mathbb{E} K(\tau) - K(0)}{\mathbb{E} K(t) - K(0)}, \qquad 0 \le \tau \le t.
		\end{align*}

		In the case of the Yule process we have $N_\beta(0)=1$ and $\mathbb{E}N_\beta(t) = e^{\beta t}$.
		Hence, conditioning
		on the number of webpages at time $t$, the
		random instants at which new webpages appear are distributed as the order statistics
		of iid random variables with distribution function
		\begin{align}
			\label{distrold}
			\mathbb{P} ( T \le \tau ) = \frac{e^{\beta \tau}-1}{e^{\beta t}-1},  \qquad 0 \le \tau \le t.
		\end{align}
		Let us now define the \emph{size} of a webpage as the number of its in-links.
		Our purpose is to identify the probability distribution of the size of a webpage extracted uniformly at random at time $t$.
		In order to do so, denote this random quantity by $\mathcal{N}^Y_t$ and call $\mathcal{N}^Y
		= \lim_{t \to \infty}\mathcal{N}^Y_t$.
		By conditioning on the random creation time $T$ of the webpage, we can write that
		\begin{align}
			\label{lenovo}
			\mathbb{P} (\mathcal{N}^Y_t = n )
			& = \mathbb{E}_T \mathbb{P} (N_\lambda(t)=n|N_\lambda(T)=1) \\
			& = \int_0^t e^{-\lambda(t-\tau)} (1-e^{-\lambda(t-\tau)})^{n-1} \frac{\beta e^{\beta \tau}}{e^{\beta t}-1} \textup{d} \tau\notag\\
			& = \frac{\beta}{1-e^{-\beta t}} \int_0^t e^{-\beta y} e^{-\lambda y} (1-e^{-\lambda y})^{n-1} \textup{d}y,
			\qquad n \ge 1, \: t \ge 0. \notag
		\end{align}
		In the limit for $t \to \infty$ we clearly obtain
		\begin{align}
			\label{two}
			\mathbb{P} ( \mathcal{N}^Y = n ) & = \int_0^\infty \beta e^{-\beta y} e^{-\lambda y}
			(1-e^{-\lambda y})^{n-1} \textup{d}y, \qquad n \ge 1,
		\end{align}
		which leads after some steps to
		\begin{align}
			\label{yule}
			\mathbb{P}(\mathcal{N}^Y=n) = \frac{\beta}{\lambda} \frac{\Gamma(n)\Gamma
			\left( 1+ \frac{\beta}{\lambda} \right)}{\Gamma \left( n+1+\frac{\beta}{\lambda} \right)}, \qquad n \ge 1.
		\end{align}
		Distribution \eqref{yule} is known as Yule or Yule--Simon distribution \cite{bala}.
		
		In the next section we will present a generalization of the classical Yule model in which the Yule processes modelling the evolution
		of the number of in-links for each webpage are substituted by fractional nonlinear birth processes (see the Appendix).
		Briefly, this corresponds
		to replacing formula \eqref{bai} with the state probabilities \eqref{nlinearnu}.

	\section{Generalized Nonlinear Yule Model}

		As outlined in the introductory section, our aim is to introduce a generalization of the Yule model presented in the preceding
		section by allowing non-Markov dependence and full nonlinear rates for the birth process governing the developing
		of the in-links. Let us thus consider a Yule-like model composed by 
		\begin{itemize}
			\item a homogeneous Yule process of rate $\beta>0$ for the development of webpages; 
			\item independent copies of a fractional nonlinear birth process $\mathfrak{N}^\nu(t)$, $t \ge 0$,
				of parameter of fractionality $\nu \in (0,1)$,
				for the development of in-links for each webpage;
			\item rates $\lambda_k$, $k=1,2,\dots$, for the fractional nonlinear birth process such that 
				explosions are not allowed, that is we admit only a finite number of jumps for any finite time. For this it is
				sufficient to assume that $\sum_{k=1}^\infty \lambda_k^{-1}=\infty$.
		\end{itemize}
		The basic properties of the fractional nonlinear birth process are recalled in the Appendix \ref{appe}.
		For a quick comparison with the classical nonlinear birth process see
		Table \ref{ipad}.
		
		Let us now call ${}_t \textswab{N}^\nu$ the number of in-links
		of a webpage chosen uniformly at random at time $t$ and define $\textswab{N}^\nu
		= \lim_{t \to \infty}{}_t \textswab{N}^\nu$.
		For this generalized Yule model we will evaluate the explicit distribution of ${}_t \textswab{N}^\nu$
		and of $\textswab{N}^\nu$ in Section \ref{distr} and their mean value in Section \ref{mee}.
		Two specific examples with rates allowing saturation are analyzed in Section \ref{satura}.

		\subsection{Distribution of $\bm{{}_t \textswab{N}^\nu$}}		

			\label{distr}		
			With the convention here and in the rest of the paper that
			empty products equal unity, we can evaluate the distribution of the random number of in-links ${}_t \textswab{N}^\nu$
			by conditioning on the random instant of time at which a webpage is created (and therefore the in-links process
			of that webpage begins).
			
			We make use here of the theoretical results contained in the papers \citep{neuts,crump,feigin,puri}.
			Conditioning on the number of webpages
			present at the observation time $t$, the random instants at which new webpages
			are created are distributed as the order statistics of iid random
			variables with distribution function
			\begin{align}
				\mathbb{P} (\mathcal{T}\le y) = \frac{\textup{e}^{\beta y}-1}{\textup{e}^{\beta t}-1}, \qquad y \in [0,t].
			\end{align}
			Hence $\mathcal{Q} = t-\mathcal{T}$, the random evolution time of the conditioned
			fractional nonlinear birth process $\mathfrak{N}^\nu(t)$, $t \ge 0$, is distributed as a truncated exponential random variable.
			It immediately follows that the distribution of ${}_t \textswab{N}^\nu$ can be determined by randomization with respect to
			$\mathcal{T}$:
			\begin{align}
				\label{randomization}
				\mathbb{P}({}_t \textswab{N}^\nu=n)
				=\mathbb{E}_{\mathcal{T}}\mathbb{P}\bigl( \mathfrak{N}^\nu(t)=n | \mathfrak{N}^\nu(\mathcal{T})=1 \bigr), \qquad n \ge 1.
			\end{align}
			Notice that the initial condition $\mathfrak{N}^\nu(\mathcal{T})=1$ for the process modelling the
			development of in-links is chosen here for consistency with the classical formulation of the Yule model
			(see \citet{yule}).
			Within the framework of network growth we are implicitely assuming that each new webpage is created
			with only one in-link.
			From a mathematical point of view, the generalization in which the fractional nonlinear birth process
			starts with $k$ individuals is straightforward. For the state probabilities,
			for instance, one should consider the general form (compare with \eqref{nlinearnu})
			\begin{equation*}
				\mathbb{P}(\bar{\mathfrak{N}}^\nu(t)=n)=
				\prod_{j=k}^{n-1} \lambda_j
				\sum_{m=k}^n \frac{ E_{\nu} (-
				\lambda_m t^\nu)}{\prod_{
				l=k,l \neq m}^n \left( \lambda_l -
				\lambda_m \right) }, \qquad n \ge k, \: t\ge 0.
			\end{equation*}

			Working out formula \eqref{randomization} we obtain
			\begin{align}
				\label{tiger}
				\mathbb{P}({}_t \textswab{N}^\nu=n) & = \frac{\beta}{1-e^{-\beta t}} \int_0^t e^{-\beta y}
				\mathbb{P}(\mathfrak{N}^\nu(y)=n)\, \textup{d}y \\
				& = \frac{1}{1-e^{-\beta t}}
				\left[ \beta^{\nu} \frac{\prod_{r=1}^{n-1}\lambda_r}{\prod_{r=1}^n(\beta^\nu+\lambda_r)}
				- \beta\int_t^\infty e^{-\beta y} \mathbb{P}(\mathfrak{N}^\nu(y)=n)\, \textup{d}y \right]. \notag
			\end{align}
			In the last step we used the explicit form of the Laplace transform
			of the state probabilities of the fractional nonlinear birth process (formula \eqref{gg}).
			Note that in the last line of \eqref{tiger} we have actually separated the limiting
			state distribution and the correction for a finite time $t$. Indeed by letting $t \to \infty$ the limiting distribution reads
			\begin{align}
				\label{pb}
				\mathbb{P}(\textswab{N}^\nu=n) = \beta^\nu
				\frac{\prod_{r=1}^{n-1}\lambda_r}{\prod_{r=1}^n(\beta^\nu+\lambda_r)}, \qquad n \ge 1.
			\end{align}
			The above formula \eqref{pb} can be reparametrized by setting $\rho^{-1}_r=\beta^\nu/\lambda_r$, obtaining
			\begin{align}
				\label{rip}
				\mathbb{P}(\textswab{N}^\nu=n) =
				\frac{\rho_n^{-1}}{\prod_{r=1}^n(\rho_r^{-1}+1)}, \qquad n \ge 1.
			\end{align}
		
			When the rates are all different we can further work out \eqref{tiger} as follows.
			Let us first evaluate the lower and upper incomplete Laplace transforms of the Mittag--Leffler function $E_\nu$
			(see the Appendix \ref{appe}, formula \eqref{ml} for the definition of the Mittag--Leffler function),
			\begin{align}
				\label{nuo}
				L = \int_0^t e^{-\beta y} E_\nu(-\lambda y^\nu) \, \textup{d} y
				= \beta^{-1}\sum_{r=0}^\infty \frac{\gamma(\nu r+1, \beta t)}{\Gamma(\nu r+ 1)}\left(-\frac{\lambda}{\beta^\nu}\right)^r,
			\end{align}		
			\begin{align}
				\label{nuo2}
				U & = \int_t^\infty e^{-\beta y} E_\nu(-\lambda y^\nu) \, \textup{d} y
				= \beta^{-1}\sum_{r=0}^\infty \frac{\Gamma(\nu r+1, \beta t)}{\Gamma(\nu r+ 1)}\left(-\frac{\lambda}{\beta^\nu}\right)^r,
			\end{align}
			where $\gamma(a,b)$ and $\Gamma(a,b)$ are respectively the lower and upper incomplete Gamma functions
			(see formulae 6.5.2 and 6.5.3 of \citet{abramowitz}), and $\beta>0$.
			By using the state probabilities distribution of the fractional nonlinear birth process (Appendix \ref{appe}, formula \eqref{nlinearnu})
			and formula \eqref{nuo},
			the distribution \eqref{tiger} (first line) can be written as
			\begin{align}
				\label{mm}
				\mathbb{P}({}_t \textswab{N}^\nu=n) =
				\frac{1}{1-e^{-\beta t}} \prod_{h=1}^{n-1} \lambda_h \sum_{m=1}^n \frac{1}{\prod_{l=1,l\ne m}^n(\lambda_l-\lambda_m)}
				\sum_{r=0}^\infty \frac{\gamma(\nu r+1,\beta t)}{\Gamma(\nu r +1)}
				\left( -\frac{\lambda_m}{\beta^\nu} \right)^r, \qquad n \ge 1,
			\end{align}
			or, equivalently, by using formula \eqref{nuo2} and the last line of \eqref{tiger}, that is highlighting the asymptotics, as
			\begin{align}
				\mathbb{P} ({}_t \textswab{N}^\nu=n) =\frac{1}{1-e^{-\beta t}}
				\prod_{h=1}^{n-1}\lambda_h \sum_{m=1}^n \frac{1}{\prod_{l=1,l\ne m}^n(\lambda_l-\lambda_m)}
				\left[ \frac{\beta^{\nu}}{\beta^\nu+\lambda_m}
				- \sum_{r=0}^\infty \frac{\Gamma(\nu r+1, \beta t)}{\Gamma(\nu r+ 1)}\left(-\frac{\lambda_m}{\beta^\nu}\right)^r \right],
				\qquad n \ge 1. \notag
			\end{align}
			
			In the next three remarks we examine three different specific cases of interest. In Remark \ref{rem1} fractionality is suppressed
			by considering $\nu=1$. In this case the model is a generalized Yule model in which a nonlinear birth process
			represents the in-links growth processes. In Remark \ref{rem2} instead we retain fractionality ($\nu \in (0,1)$) but
			with linear rates ($\lambda_r = \lambda r$). Remark \ref{rem3} recovers the known form of the in-degree distribution of a webpage
			chosen uniformly at random in the classical Yule model for any fixed time $t$ as a special case of our more general 
			formula.
			
			\begin{remark}
				\label{rem1}
				In the non fractional case, that is $\nu = 1$, we can further simplify the above probability mass function \eqref{mm}.
				Indeed, recalling that $E_1 (-\lambda_my) = e^{-\lambda_my}$ and that its lower incomplete Laplace transform
				reads $\int_0^t e^{-\beta y} e^{-\lambda_m y} \textup{d}y = (1-e^{-t(\beta+\lambda_m)})/(\beta+\lambda_m)$,
				we obtain,
				\begin{align}
					\mathbb{P}({}_t \textswab{N}^1=n) =
					\frac{\beta}{1-e^{-\beta t}} \prod_{r=1}^{n-1} \lambda_r \sum_{m=1}^n \frac{1}{\prod_{l=1,l\ne m}^n(\lambda_l-\lambda_m)}
					\left( \frac{1-e^{-(\beta+\lambda_m)t}}{\beta+\lambda_m} \right), \qquad n \ge 1,					
				\end{align}
			\end{remark}
						
			\begin{remark}
				\label{rem2}
				For linear rates, $\lambda_r = \lambda r$, $r \ge 1$, the fractional nonlinear birth process
				$\mathfrak{N}^\nu(t)$ coincides with the fractional Yule
				process $\mathfrak{N}^\nu_{\text{lin}}(t)$ (see the Appendix \ref{appe}).
				The distribution \eqref{tiger} can be written, for $n \ge 1$, as
				\begin{align}
					\label{tiger2}
					\mathbb{P}({}_t \textswab{N}_{\text{lin}}^\nu=n)
					& = \frac{\beta}{1-e^{-\beta t}} \left[ \int_0^\infty e^{-\beta y} \mathbb{P}(\mathfrak{N}^\nu_{\text{lin}}(y)=n)\, \textup{d}y
					- \int_t^\infty e^{-\beta y} \mathbb{P}(\mathfrak{N}^\nu_{\text{lin}}(y)=n)\, \textup{d}y \right].
				\end{align}
				By making use of the simpler form of the state probability distribution of the fractional Yule process
				(see the Appendix \ref{appe},
				formula \eqref{gengeo}) and the Laplace transform of the Mittag--Leffler function \eqref{mitl} we obtain that
				\begin{align}
					& \int_0^\infty e^{-\beta y} \mathbb{P}(\mathfrak{N}^\nu_{\text{lin}}(y)=n) \, \textup{d} y
					= \int_0^\infty e^{-\beta y} \sum_{j=1}^n \binom{n-1}{j-1} (-1)^{j-1} E_\nu(-\lambda jy^\nu) \, \textup{d} y\\
					& = \sum_{j=1}^n \binom{n-1}{j-1}(-1)^{j-1} \frac{\beta^{\nu-1}}{\beta^\nu+\lambda j}
					= \beta^{\nu-1} \sum_{j=1}^n \binom{n-1}{j-1}(-1)^{j-1} \int_0^\infty e^{-w(\beta^\nu+\lambda j)}\textup{d}w \notag \\
					& = \beta^{\nu-1} \int_0^\infty e^{-w(\beta^\nu + \lambda)} \sum_{j=0}^{n-1} \binom{n-1}{j} (-1)^j e^{-w\lambda j}
					= \beta^{\nu-1} \int_0^\infty e^{-w \beta^\nu} e^{-w\lambda} (1-e^{-w \lambda})^{n-1} \textup{d}w \notag \\
					& = \beta^{\nu-1} \int_0^1 e^{\frac{\beta^\nu}{\lambda}\log y} (1-y)^{n-1}\frac{\textup{d}y}{\lambda}
					= \frac{\beta^{\nu-1}}{\lambda} \int_0^1 y^{\frac{\beta^\nu}{\lambda}}(1-y)^{n-1}\, \textup{d} y \notag
					= \frac{\beta^{\nu-1}}{\lambda}\frac{\Gamma
					\left( \frac{\beta^\nu}{\lambda}+1 \right) \Gamma(n)}{\Gamma\left( \frac{\beta^\nu}{\lambda}
					+n+1 \right)}. \notag
				\end{align}
				Therefore we can write the distribution of interest for each finite time $t$ as
				\begin{align}
					\label{water}
					\mathbb{P}({}_t \textswab{N}_{\text{lin}}^\nu=n)
					& = \frac{\beta}{1-e^{-\beta t}} \left[ \frac{\beta^{\nu-1}}{\lambda}\frac{\Gamma\left(
					\frac{\beta^\nu}{\lambda}+1 \right) \Gamma(n)}{\Gamma\left( \frac{\beta^\nu}{\lambda}+n+1 \right)}
					- \int_t^\infty e^{-\beta y} \mathbb{P}(\mathfrak{N}^\nu_{\text{lin}}(y)=n)\, \textup{d}y \right]\\
					& = \frac{1}{1-e^{-\beta t}} \left[ \frac{\beta^\nu}{\lambda}\frac{\Gamma\left(
					\frac{\beta^\nu}{\lambda}+1 \right) \Gamma(n)}{\Gamma\left( \frac{\beta^\nu}{\lambda}+n+1 \right)}
					- \sum_{j=1}^n \binom{n-1}{j-1} (-1)^{j-1}
					\sum_{r=0}^\infty \frac{\Gamma(\nu r+1,\beta t)}{\Gamma(\nu r+1)} \left( -\frac{\lambda}{\beta^\nu} m \right)^r \right]\notag,
				\end{align}
				and the limiting distribution as
				\begin{align}
					\label{cup}
					\mathbb{P}(\textswab{N}_{\text{lin}}^\nu=n) = \lim_{t \to \infty} \mathbb{P}({}_t\textswab{N}_{\text{lin}}^\nu=n)
					= \frac{\beta^\nu}{\lambda}\frac{\Gamma\left( \frac{\beta^\nu}{\lambda}+1 \right) \Gamma(n)}{\Gamma\left(
					\frac{\beta^\nu}{\lambda}
					+n+1 \right)}, \qquad n \ge 1.
				\end{align}
				The probability mass function \eqref{cup} is the usual Yule or Yule--Simon distribution
				of parameter $\rho^{-1}=\beta^\nu/\lambda$. We remark that from \eqref{water} the
				contribution to the asymptotics and the finite time correction are clear.
			\end{remark}
		
			\begin{remark}
				\label{rem3}
				If $\nu=1$, equation \eqref{water} gives us the probability distribution of ${}_t \textswab{N}_{\text{lin}}^1$, i.e.\ that
				of the classical Yule model for a finite time $t$. Indeed we have that, for $n \ge 1$,
				\begin{align}
					\mathbb{P}({}_t \textswab{N}_{\text{lin}}^1=n) =
					\frac{1}{1-e^{-\beta t}} \left[ \frac{\beta}{\lambda}\frac{\Gamma\left(
					\frac{\beta}{\lambda}+1 \right) \Gamma(n)}{\Gamma\left( \frac{\beta}{\lambda}+n+1 \right)}
					- \sum_{j=1}^n \binom{n-1}{j-1} (-1)^{j-1}
					\sum_{r=0}^\infty \frac{\Gamma(r+1,\beta t)}{r!} \left( -\frac{\lambda}{\beta} m \right)^r \right]\notag,
				\end{align}
				or, in a more compact form,
				\begin{align}
					\mathbb{P}({}_t \textswab{N}_{\text{lin}}^1=n) & =
					\frac{\beta}{1-e^{-\beta t}} \sum_{j=1}^n \binom{n-1}{j-1} (-1)^{j-1} \frac{1-e^{-t(\beta + \lambda j)}}{\beta + \lambda j} \\
					& = \frac{1}{1-e^{-\beta t}}\frac{\beta}{\lambda} \left[ \frac{\Gamma\left(
					\frac{\beta}{\lambda}+1 \right) \Gamma(n)}{\Gamma\left( \frac{\beta}{\lambda}+n+1 \right)}
					- \textup{Be}\left( e^{-\lambda t}; \frac{\beta}{\lambda}+1,n \right) \right]
					= \frac{1}{1-e^{-\beta t}}\frac{\beta}{\lambda} \textup{Be}\left( 1-e^{-\lambda t} ; \frac{\beta}{\lambda}+1,n \right).
					\notag
				\end{align}			
				where $\textup{Be}(z;a,b)$ is the incomplete Beta function.
			\end{remark}

		\subsection{Mean of $\bm{{}_t \textswab{N}^\nu}$}
		
			\label{mee}		
			In this section we derive the explicit form of $\mathbb{E} \, {}_t\textswab{N}^\nu$, that is
			the expected value of the number of in-links of a webpage chosen
			uniformly at random from those present at time $t$ in the generalized Yule model with which we are concerned.
			By randomizing
			on the random creation time of the uniformly chosen webpage we have that
			\begin{align}
				\label{ps}
				\mathbb{E} \, {}_t\textswab{N}^\nu = \frac{\beta}{1-e^{-\beta t}}
				\int_0^t e^{-\beta y} \mathbb{E} \, \mathfrak{N}^\nu(y) \, \textup{d}y, \qquad t \ge 0.
			\end{align}
			In order to obtain an explicit expression we now need to use an explicit form for the expected number
			of in-links in the fractional nonlinear birth process $\mathfrak{N}^\nu$.
			To do so we make use of Theorem 3.2 of \citet{MR3021490} where the mean value of the fractional
			nonlinear birth process in the case of rates all different is derived:
			\begin{align}
				\label{ps2}
				\mathbb{E} \, \mathfrak{N}^\nu(t) = 1 + \sum_{k=1}^\infty \left\{ 1-\sum_{m=1}^k \left( \prod_{\substack{l=1\\l \ne m}}^k
				\frac{\lambda_l}{\lambda_l-\lambda_m} \right) E_\nu (-\lambda_m t^\nu) \right\}, \qquad t \ge 0.
			\end{align}
			In the next theorem we calculate, for any fixed time $t$, the expected value of the number of in-links of a webpage chosen
			uniformly at random from those present at time $t$.
		
			\begin{theorem}
				If the rates $\lambda_r$, $r \ge 1$, are all different we have, for $t \ge 0$,
				\begin{align}
					\label{lubos}
					\mathbb{E} \, {}_t\textswab{N}^\nu =
					1 + \sum_{k=1}^\infty \left[ 1- \frac{1}{1-e^{-\beta t}} \sum_{m=1}^k
					\left( \prod_{\substack{l=1\\l \ne m}}^k \frac{\lambda_l}{\lambda_l-\lambda_m} \right)
					\sum_{r=0}^\infty \frac{\gamma(\nu r+1,\beta t)}{\Gamma(\nu r +1)} \left( -\frac{\lambda_m}{\beta^\nu} \right)^r \right].
				\end{align}
			\end{theorem}
		
			\begin{proof}
				It follows naturally by inserting \eqref{ps2} into \eqref{ps} as follows.
				\begin{align}
					\mathbb{E} \, {}_t\textswab{N}^\nu
					& = \frac{\beta}{1-e^{-\beta t}} \int_0^t e^{-\beta y} \left[ 1+\sum_{k=1}^\infty 
					\left\{ 1-\sum_{m=1}^k \left( \prod_{\substack{l=1\\l \ne m}}^k
					\frac{\lambda_l}{\lambda_l-\lambda_m} \right) E_\nu (-\lambda_m y^\nu) \right\} \right] \textup{d}y \\
					& = \frac{\beta}{1-e^{-\beta t}} \left[ \int_0^t e^{-\beta y} \textup{d}y +
					\sum_{k=1}^\infty \left\{ \int_0^t e^{-\beta y} \textup{d}y
					- \sum_{m=1}^k \left( \prod_{\substack{l=1\\l \ne m}}^k
					\frac{\lambda_l}{\lambda_l-\lambda_m} \right) \int_0^t e^{-\beta y}E_\nu (-\lambda_m t^\nu) \textup{d}y \right\} \right] \notag \\
					& = 1 + \sum_{k=1}^\infty \left[ 1- \frac{1}{1-e^{-\beta t}} \sum_{m=1}^k
					\left( \prod_{\substack{l=1\\l \ne m}}^k \frac{\lambda_l}{\lambda_l-\lambda_m} \right)
					\sum_{r=0}^\infty \frac{\gamma(\nu r+1,\beta t)}{\Gamma(\nu r +1)} \left( -\frac{\lambda_m}{\beta^\nu} \right)^r \right]. \notag
				\end{align}
			\end{proof}
		
			\begin{remark}
				When $\nu=1$, that is in the classical non-fractional case, the mean value simplifies to
				\begin{align}
					\mathbb{E} \, {}_t\textswab{N}^1 =
					1 + \sum_{k=1}^\infty \left[ 1- \frac{\beta}{1-e^{-\beta t}} \sum_{m=1}^k
					\left( \prod_{\substack{l=1\\l \ne m}}^k \frac{\lambda_l}{\lambda_l-\lambda_m} \right)
					\frac{1-e^{-(\beta +\lambda_m)t}}{\beta + \lambda_m} \right].
				\end{align}
			\end{remark}
		
			\begin{remark}
				The expected value of the limiting random variable $\textswab{N}^\nu$ can be determined directly from
				\eqref{lubos} as
				\begin{align}
					\label{lubos2}
					\mathbb{E} \, \textswab{N}^\nu = 1+ \sum_{k=1}^\infty
					\left[ 1-\sum_{m=1}^k \left( \prod_{\substack{l=1\\l \ne m}}^k \frac{\lambda_l}{\lambda_l-\lambda_m} \right)
					\frac{\beta^\nu}{\beta^\nu+\lambda_m} \right].
				\end{align}
			\end{remark}
			
			It is interesting to notice from \eqref{lubos} and \eqref{lubos2} that the effect of the parameter $\nu$ which is
			present in the mean value \eqref{lubos} for any fixed time $t$, determines a change in the parametrization
			in the limiting mean value \eqref{lubos2}.

		\subsection{Examples: models with saturation}
		
			\label{satura}
			The general model depicted above can be made more specific seeking for particular properties. For example
			an interesting behaviour that could possibly lead to a more realistic scenario is when the number of in-links
			for each webpage has intrinsically a fixed value to which it saturates.
			A quite general saturating behaviour can be achieved by truncating the rates at $N-1$ (so that $\lambda_N=0$). This can be done
			by considering the rates as the weights of a discrete finite measure on the finite set $\{ 1,2,\dots,N-1 \}$.
			A further rather general model admitting a saturating behaviour is the one in which the rates specialize as
			\begin{align}
				\lambda_j= \eta \left(\frac{j}{N}\right)^{\omega_1}
				\left(\frac{N-j}{N}\right)^{\omega_2} =\lambda j^{\omega_1}(N-j)^{\omega_2},
				\qquad (\omega_1,\omega_2) \in [0,\infty)\times (0,\infty), \: \eta>0,
			\end{align}			
			where $\lambda= \eta/N^{\omega_1+\omega_2}$.
			These rates clearly do not
			allow explosion in a finite time. Note that we have explicitly excluded the cases in which $\omega_2 = 0$ as this choice
			implies unbounded growth.
			By specializing the rates in \eqref{pb}, and considering $1 \le n \le N$, we have
			\begin{align}
				\mathbb{P} (\textswab{N}^\nu=n) & = \frac{\beta^\nu \prod_{r=1}^{n-1}(\lambda r^{\omega_1}(N-r)^{\omega_2})}{
				\prod_{r=1}^n (\beta^\nu+\lambda r^{\omega_1}(N-r)^{\omega_2})}
				= \frac{(N-1)!^{\omega_2}}{(N-n)!^{\omega_2}}
				\frac{\beta^\nu \lambda^{n-1} (n-1)!^{\omega_1}}{\beta^{\nu n} \prod_{r=1}^n
				\left( 1+\frac{\lambda}{\beta^\nu}r^{\omega_1}(N-r)^{\omega_2} \right)} \\
				& = \left( \frac{\lambda}{\beta^\nu} \right)^{n-1} \frac{\Gamma^{\omega_1}(n) \Gamma^{\omega_2}(N)}{\Gamma^{\omega_2}(N-n+1)}
				\frac{1}{\prod_{r=1}^n \left( 1+\frac{\lambda}{\beta^\nu}r^{\omega_1}(N-r)^{\omega_2} \right)} \notag \\
				& = \rho^{n-1} \frac{\Gamma^{\omega_1}(n) \Gamma^{\omega_2}(N)}{\Gamma^{\omega_2}(N-n+1)}
				\frac{1}{\prod_{r=1}^n \left( 1+\rho r^{\omega_1}(N-r)^{\omega_2} \right)}. \notag
			\end{align}

			\subsubsection{First example}

				\begin{figure}
					\centering
					\includegraphics[scale=.42]{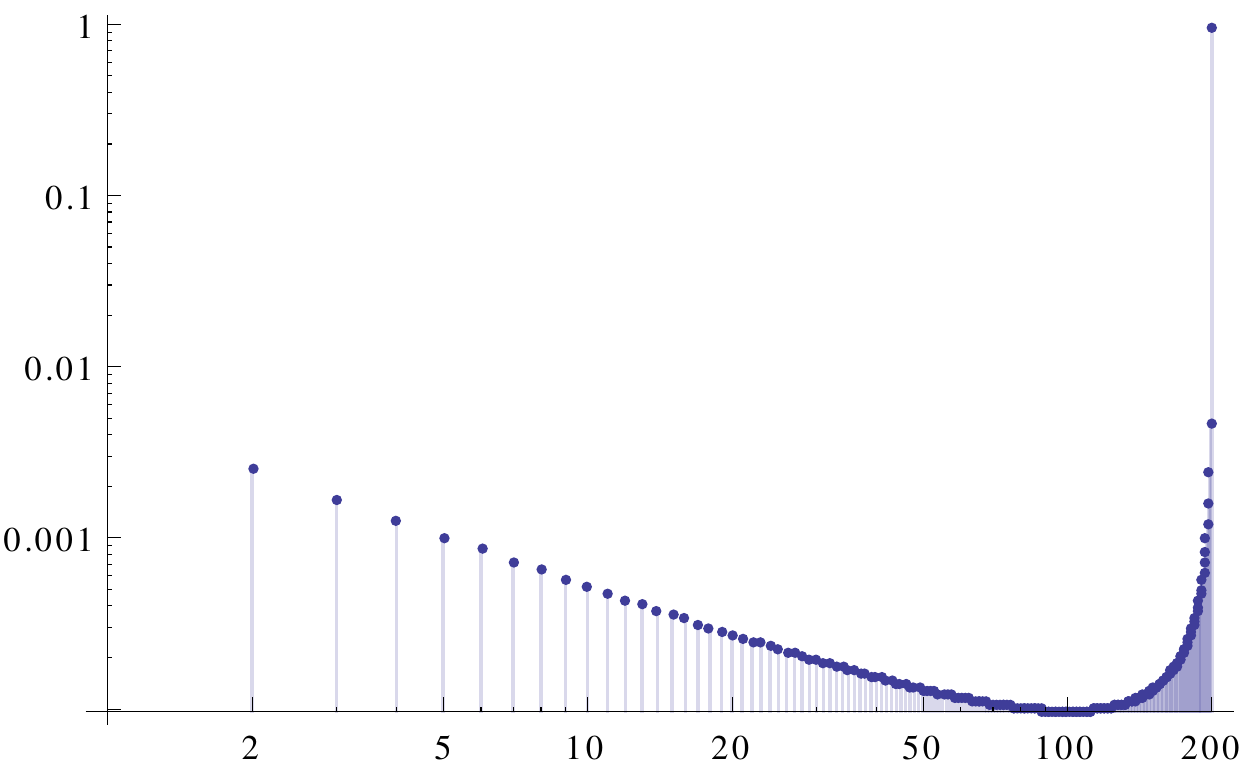}
					\includegraphics[scale=.42]{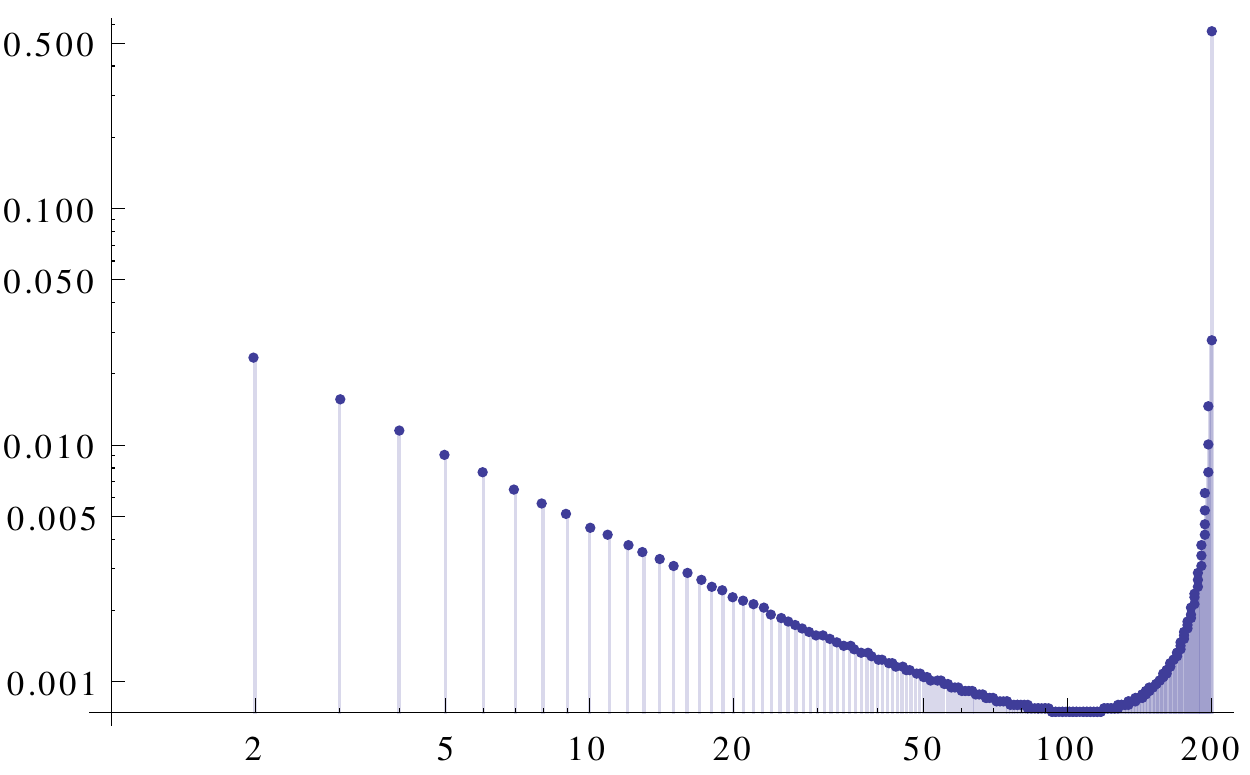}
					\includegraphics[scale=.42]{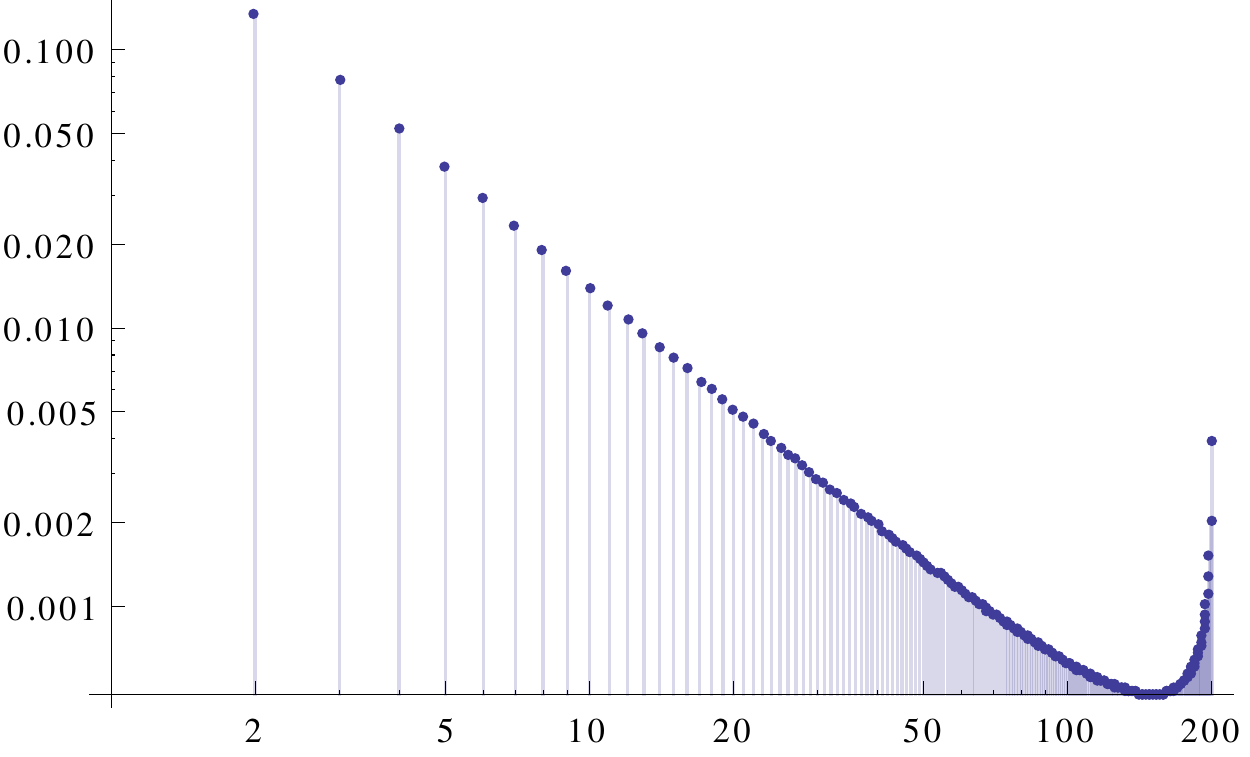}
					\includegraphics[scale=.42]{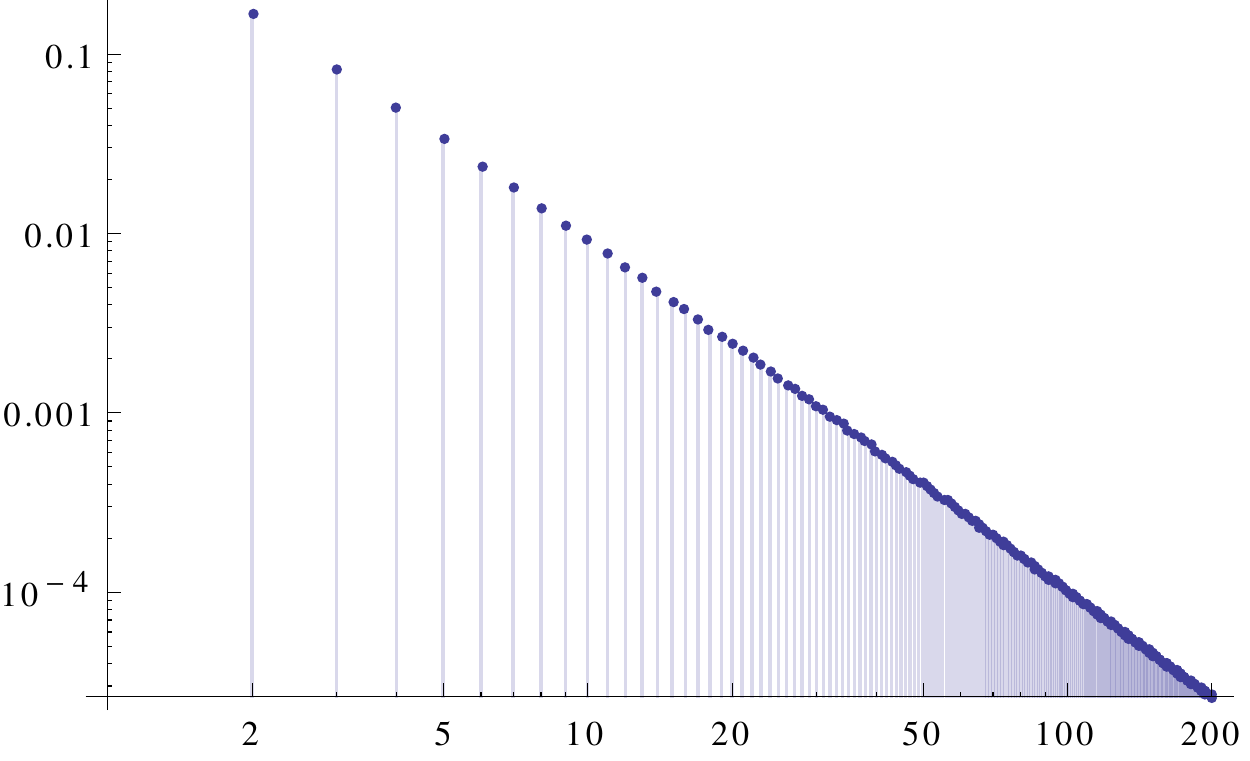}
					\includegraphics[scale=.42]{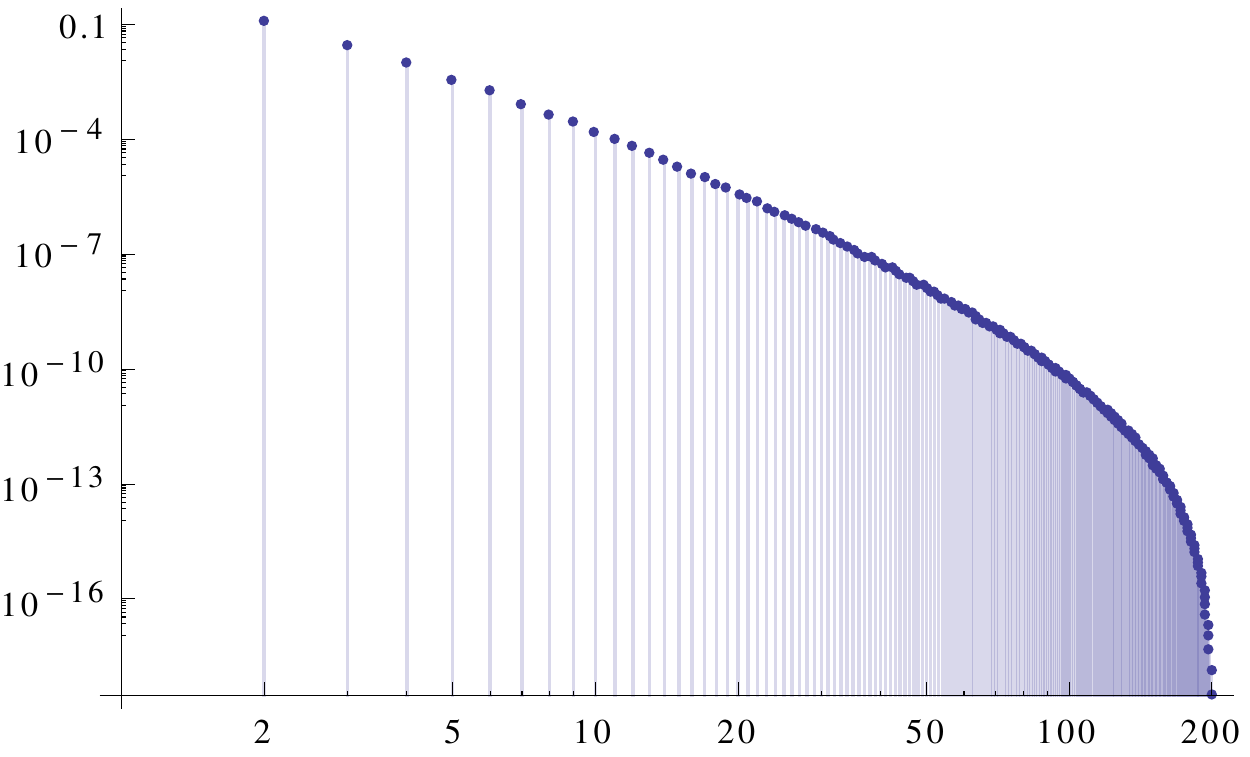}
					\includegraphics[scale=.42]{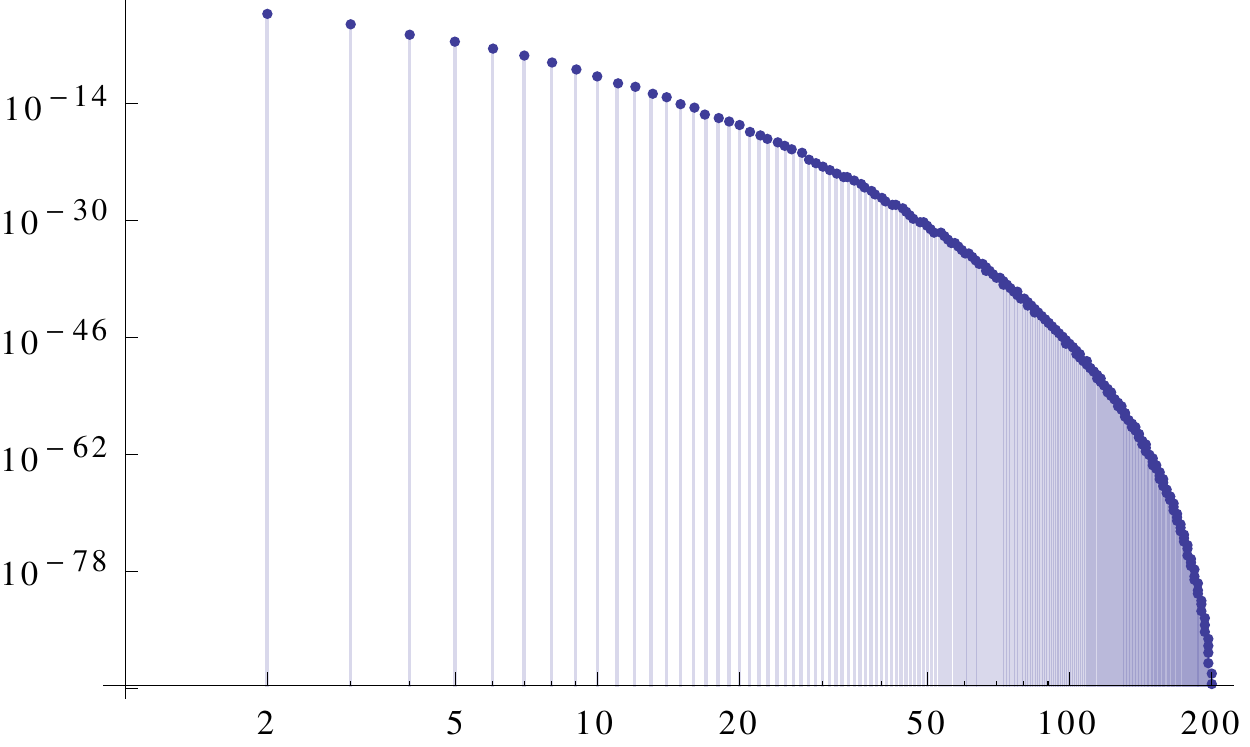}
					\caption{\label{clack}Various plots of $\mathbb{P}(\textswab{N}_{s_2}^\nu=n)$. From left to right, top to bottom, the parameters
						are respectively $(\rho,N)= \{ (1,200),(0.1,200),(0.01,200),(0.005,200),(0.001,200),(0.0001,200) \}$.}
				\end{figure}
							
				A first example is when the nonlinear rates $\lambda_j$ particularize to $\lambda_j=\lambda(N-j)$, $\lambda>0$,
				$1 \le j \le N$, where $N$
				is the threshold integer value which cannot be crossed: when the process is in state $N$ the birth rate vanishes.
				This corresponds to $(\omega_1,\omega_2)=(0,1)$, $\lambda= \eta/N$.
				In this case, since the rates are all different we can explicitly calculate the Laplace transform
				of the state distribution of the fractional nonlinear birth process governing the evolution of the in-links
				for each webpage as follows (or equivalently we suitably specialize the birth rates in formula \eqref{gg}).
				\begin{align}
					\mathbb{L}_n(z) & = \prod_{m=1}^{n-1} \lambda(N-j) \sum_{m=1}^n \frac{z^{\nu-1}}{z^\nu+\lambda(N-m)}
					\frac{1}{\prod_{l=1,l \ne m}^n (\lambda(N-l)-\lambda(N-m))} \\
					& = \sum_{j=1}^{n-1} \frac{(N-1)N \dots (N-n+1)}{(m-1)m \dots (m-m+1)(m-m-1)\dots (m-n)} \frac{z^{\nu-1}}{z^\nu+\lambda(N-m)}
					\notag \\
					& = \sum_{m=1}^n \frac{(N-1)!}{(N-n)!(m-1)!(n-m)!}(-1)^{n-m}\frac{z^{\nu-1}}{z^\nu+\lambda(N-m)} \notag \\
					& = \binom{N-1}{N-n} \sum_{m=1}^n \binom{n-1}{m-1}(-1)^{n-m} \frac{z^{\nu-1}}{z^\nu+\lambda(N-m)}, \qquad 1 \le n \le N. \notag
				\end{align}
				
				Let us now denote $\textswab{N}_{s_1}^\nu$ the size of a randomly chosen webpage for $t \to \infty$
				for this first model allowing saturation.
				We immediately have for $1 \le n \le N$,
				\begin{align}
					\mathbb{P}(\textswab{N}_{s_1}^\nu=n) = \binom{N-1}{N-n}\sum_{m=1}^n \binom{n-1}{m-1} (-1)^{n-m}
					\frac{\beta^\nu}{\beta^\nu+\lambda(N-m)},
				\end{align}
				that is to say, by using the usual parametrization $\rho=\lambda/\beta^\nu$,
				\begin{align}
					\mathbb{P}(\textswab{N}_{s_1}^\nu=n) = \binom{N-1}{N-n}\sum_{m=1}^n \binom{n-1}{m-1} (-1)^{n-m} \frac{1}{1+\rho(N-m)}.
				\end{align}
				
				\begin{remark}
					The case $(\omega_1,\omega_2) = (1,0)$ corresponds to the classical Yule model.
				\end{remark}

			\subsubsection{Second example}			

				\begin{figure}
					\centering
					\includegraphics[scale=.42]{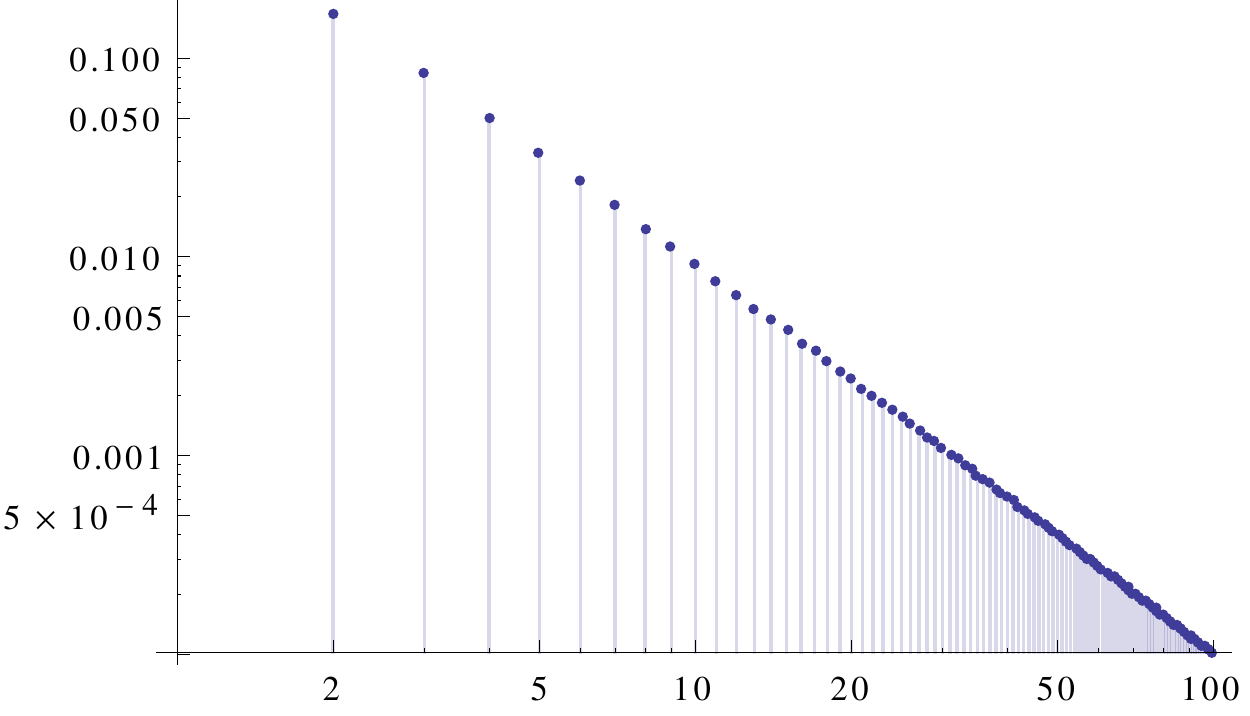}
					\includegraphics[scale=.42]{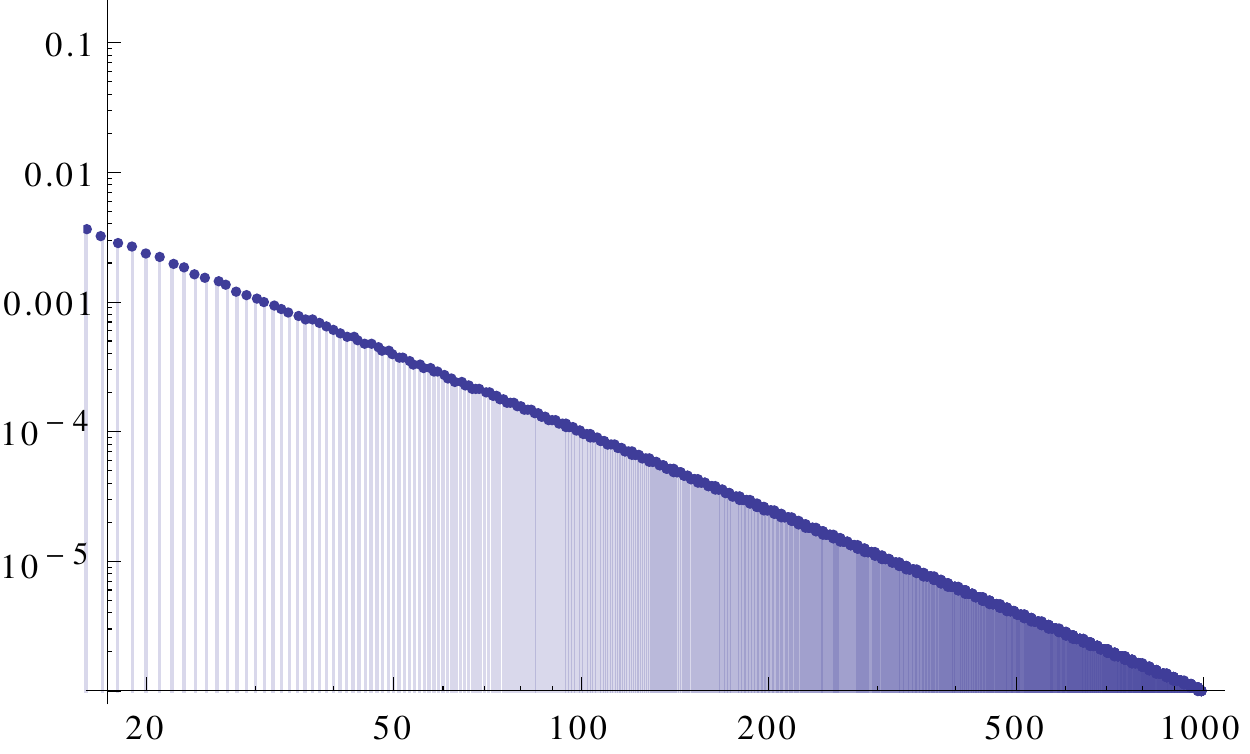}
					\includegraphics[scale=.42]{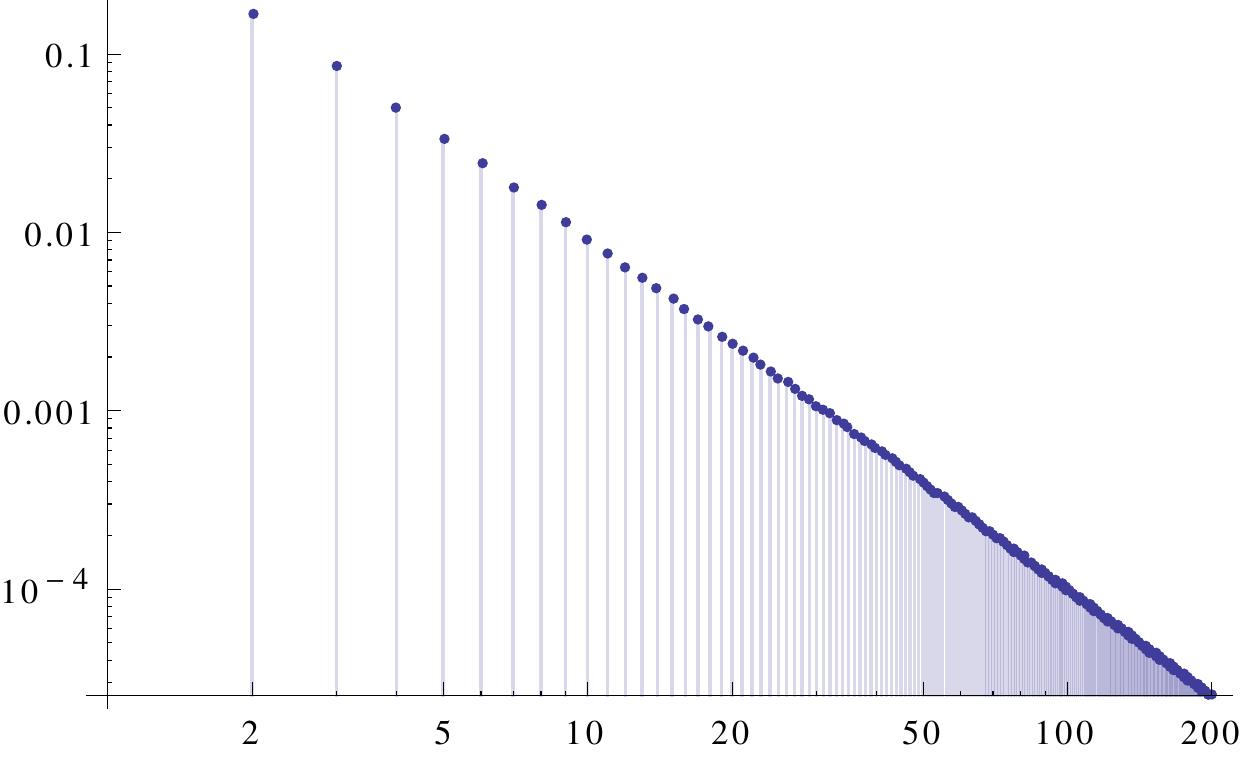}
					\caption{\label{clack2}Plots of $\mathbb{P}(\textswab{N}_{s_2}^\nu=n)$
						 when $\rho=1/(N+1)$. From left to right. The parameters are $(\rho,N)= \{ (0.01,100), (0.001,1000), (0.005,200) \}$.}
				\end{figure}

				Here we specialize the birth rates as $\lambda_r = \lambda r(N-r)$, $1 \le r \le N$.
				This corresponds to $(\omega_1,\omega_2)=(1,1)$, $\lambda= \eta/N^2$.
				If $\textswab{N}_{s_2}^\nu$ is the size of a randomly chosen webpage for $t \to \infty$
				for this second model allowing saturation, we have
				\begin{align}
					\mathbb{P} (\textswab{N}_{s_2}^\nu=n) = {} &\frac{\beta^\nu \prod_{r=1}^{n-1}(\lambda r(N-r))}{\prod_{r=1}^n (\beta^\nu+\lambda r
					(N-r))}
					= \frac{(N-1)!}{(N-n)!} \frac{\beta^\nu \lambda^{n-1} (n-1)!}{\beta^{\nu n}\prod_{r=1}^n
					\left( 1+\frac{\lambda}{\beta^\nu}r(N-r) \right)} \\
					= {} & \left( \frac{\lambda}{\beta^\nu} \right)^{n-1} \frac{\Gamma(n)\Gamma(N)}{\Gamma(N-n+1)}
					\frac{1}{\prod_{r=1}^n
					\left( 1+\frac{\lambda}{\beta^\nu}r(N-r) \right)}, \qquad 1 \le n \le N. \notag
				\end{align}
				Reparametrizing by $\rho=\lambda/\beta^\nu$, we obtain
				\begin{align}
					\label{ppp}
					\mathbb{P} (\textswab{N}_{s_2}^\nu=n) & = \rho^{n-1} \frac{\Gamma(n) \Gamma(N)}{\Gamma(N-n+1)}
					\frac{1}{\prod_{r=1}^n (1+\rho r (N-r))}, \qquad 1 \le n \le N.
				\end{align}
				Let us now simplify the product in \eqref{ppp} as follows.
				\begin{align}
					& \prod_{r=1}^n (1+\rho r (N-r)) = \prod_{r=1}^n (1-\rho r^2 + r \rho N)
					= (-\rho)^n \prod_{r=1}^n \left( r^2-2 r \frac{N}{2} - \frac{1}{\rho} \right)  \\
					& = (-\rho)^n \prod_{r=1}^n \left( r^2 + \frac{N^2}{4} -2 r \frac{N}{2} -\frac{N^2}{4} -\frac{1}{\rho} \right)
					= (-\rho)^n \prod_{r=1}^n \left[ \left( 1-\frac{N}{2} + r -1 \right)^2 -\frac{1}{4} \left( N^2+\frac{4}{\rho} \right) \right].
					\notag
				\end{align}
				By making the substitutions $a=1-N/2$ and $b=(N^2+4/\rho)^{1/2}/2$ we have
				\begin{align}
					\prod_{r=1}^n (1+\rho r (N-r)) & = (-\rho)^n  \prod_{r=1}^n \left[ (a+r-1)^2 -b^2 \right]
					= (-\rho)^n \prod_{r=1}^n (a-b+r-1) \prod_{r=1}^n (a+b+r-1).
				\end{align}
				Considering that the Pochhammer symbol (e.g.\ \citet{MR2723248}, page 136) can be written also as $(c)_n = \prod_{r=1}^n(c+r-1)
				= \Gamma(c+n)/\Gamma(c)$, we obtain that
				\begin{align}
					\prod_{r=1}^n (1+\rho r (N-r)) = (-\rho)^n \left( 1-\frac{N}{2}-\frac{1}{2}\sqrt{N^2+\frac{4}{\rho}} \right)_n
					\left( 1-\frac{N}{2}+\frac{1}{2}\sqrt{N^2+\frac{4}{\rho}} \right)_n.
				\end{align}

				Concluding, for $1 \le n \le N$, the explicit expression of the probability mass function for the random number of
				in-links in a webpage chosen uniformly at random can be written as
				\begin{align}
					\label{ccc}
					\mathbb{P}(\textswab{N}_{s_2}^\nu = n) & =
					\frac{(-1)^n}{\rho} \frac{\Gamma(n)\Gamma(N)}{\Gamma(N-n+1)}
					\frac{\Gamma \left( 1-\frac{N}{2}-\frac{1}{2} \sqrt{N^2+\frac{4}{\rho}} \right)}{\Gamma \left( n+1-\frac{N}{2}
					-\frac{1}{2} \sqrt{N^2+\frac{4}{\rho}} \right)}
					\frac{\Gamma \left( 1-\frac{N}{2}+\frac{1}{2} \sqrt{N^2+\frac{4}{\rho}} \right)}{\Gamma
					\left( n+1-\frac{N}{2} +\frac{1}{2} \sqrt{N^2+\frac{4}{\rho}} \right)} \\
					& = \rho^{-1} \frac{\Gamma(n)\Gamma(N)}{\Gamma(N-n+1)}
					\frac{\Gamma \left( \frac{1}{2} \sqrt{N^2+\frac{4}{\rho}}+1-\frac{N}{2} \right)}{\Gamma \left( \frac{1}{2}
					\sqrt{N^2+\frac{4}{\rho}}+1-\frac{N}{2} +n \right)}
					\frac{\Gamma \left( \frac{1}{2} \sqrt{N^2+\frac{4}{\rho}}+\frac{N}{2}-n \right)}{\Gamma
					\left( \frac{1}{2} \sqrt{N^2+\frac{4}{\rho}}+\frac{N}{2} \right)}. \notag
				\end{align}
								
				See in Figure \ref{clack} various plots of the probability mass function $\mathbb{P}(\textswab{N}_{s_2}^\nu = n)$, $n \ge 1$
				(equation \eqref{ccc}) for different values of the characterizing parameters, $N$ (threshold at which saturation occurs),
				and $\rho$ (which takes into account the webpage appearing rate $\beta$, the in-links appearing rate $\lambda$, and the parameter of
				fractionality $\nu$). Notably, when $\rho = 1/(N+1)$ the distribution \eqref{ccc}
				simplifies to
				\begin{equation}
					\label{elv}
					\mathbb{P}(\textswab{N}_{s_2}^\nu = n) = \left( 1+ \frac{1}{N} \right) \frac{1}{n^2+n}, \qquad 1 \le n \le N.
				\end{equation}
				
				The above distribution is shown in Figure \ref{clack2} for different values of the parameter $\rho$. 
				Figure \ref{clack3} shows a graphical investigation of the asymptotic probability of selecting a
				saturated ($\mathbb{P}(\textswab{N}_{s_2}^\nu=N)$)
				or almost saturated ($\mathbb{P}(\textswab{N}_{s_2}^\nu=N-1)$)
				webpage with $\rho=1/(N-1)^\alpha$, $\alpha > 0$, with respect to the threshold $N$. First note that
				for $\alpha=1$ (the case considered in Figure \ref{clack2}) we have a perfect power-law
				behaviour as formula \eqref{elv} becomes $\mathbb{P}(\textswab{N}_{s_2}^\nu=N) = N^{-2}$, $N \ge 1$ (see Figure
				\ref{clack3}(b)). Interestingly enough, the shape and the limiting value of the probability of selecting a saturated webpage
				strongly depend on whether $\alpha$ is larger or smaller than unity.

				\begin{figure}
					\centering
					\subfigure[Plot of the limiting probability $\mathbb{P}(\textswab{N}_{s_2}^\nu=N)$, for $N \ge 1$,
						$\rho=1/(N+1)^\alpha$, $\alpha = (0.1,0.2,0.3,0.4,0.5,0.6,0.7,0.8)$
						from top to bottom.]{\includegraphics[scale=.42]{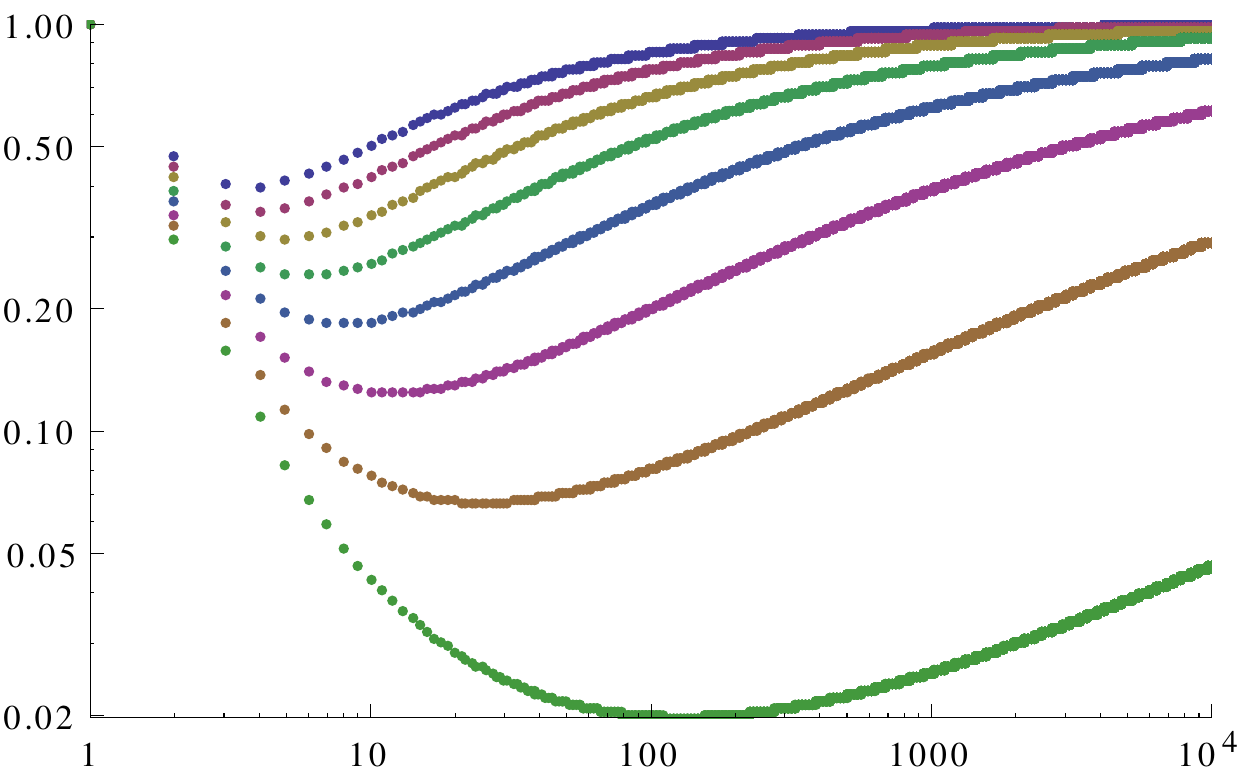}}
					\quad
					\subfigure[Plot of $\mathbb{P}(\textswab{N}_{s_2}^\nu=N)$, $N \ge 1$, $\rho=1/(N+1)^\alpha$, $\alpha = (0.8,0.9,1,1.1,1.2)$
						from top to bottom.]{\includegraphics[scale=.42]{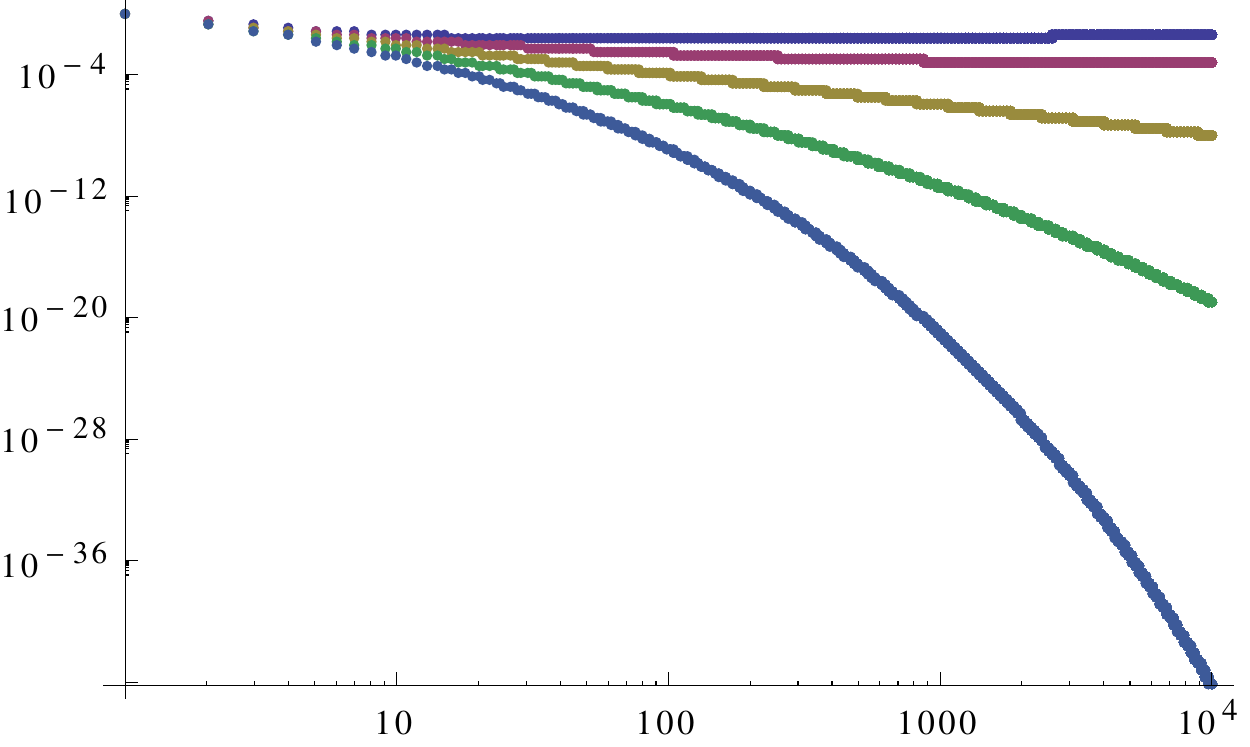}}
					\quad
					\subfigure[Plot of $\mathbb{P}(\textswab{N}_{s_2}^\nu=N-1)$, $N \ge 1$,
						$\rho=1/(N+1)^\alpha$, $\alpha = (0.1,0.2,0.3,0.4,0.5,0.6)$
						from top to bottom.]{\includegraphics[scale=.42]{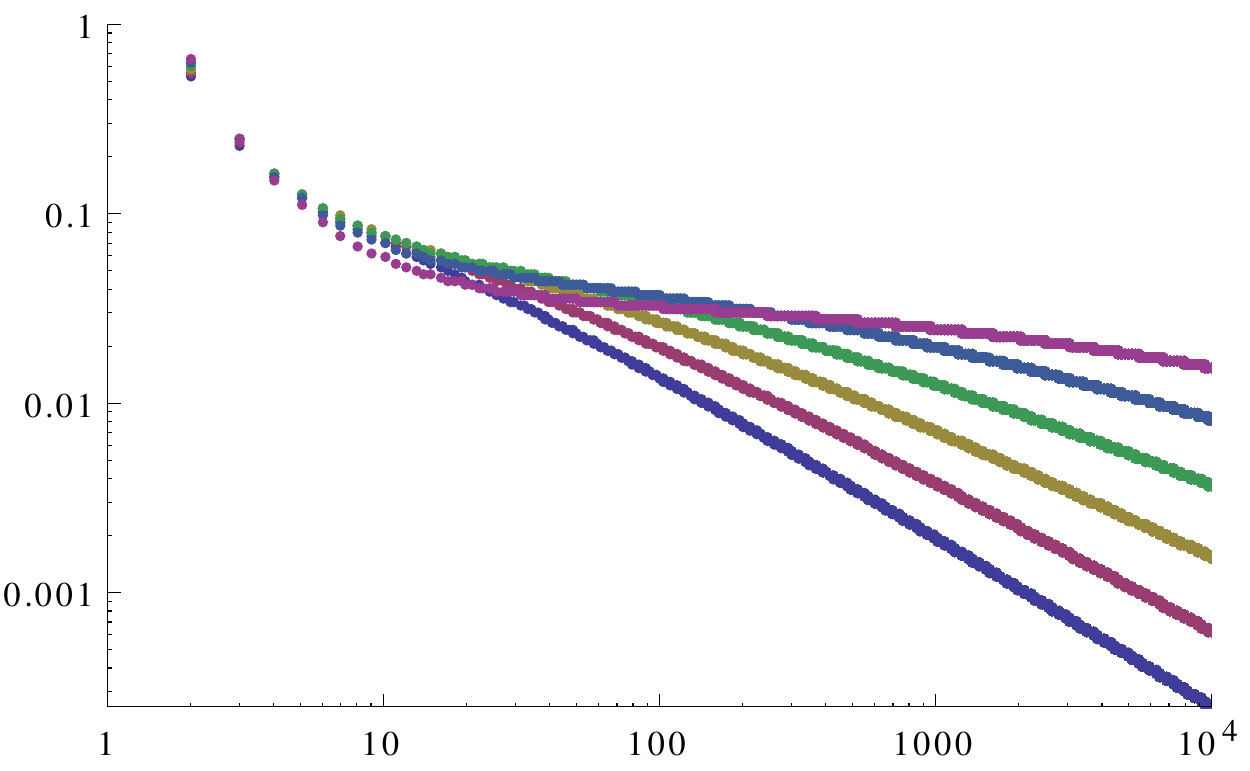}}
					\caption{\label{clack3}The limit probability of selecting a saturated webpage (figures (a) and (b)) or an
						almost saturated webpage (figure (c)) with parameter $\rho=1/(N+1)^\alpha$ for different values
						of $\alpha$. Notice how, if $\alpha \in (0,1)$, the probability tends to unity for $N \to \infty$
						(figures (a) and (b)) while for $\alpha \ge 1$ it decreases towards zero. Accordingly,
						the probability of selecting a webpage with $N-1$ in-links vanishes asymptotically for $\alpha \in (0,1)$ (figure (c)).}
				\end{figure}

	\section{Summary and Conclusions}
	
		In the paper we have developed a model which generalizes the classical Yule model still maintaining mathematical tractability.
		The presented model is interesting in that it admits a nonlinear growth for the number of in-links.
		More precisely, the generalized model is constructed by considering a fractional nonlinear birth process with completely
		general rates governing the process of creation of in-links for each webpage present at a specific time.
		The distribution of the number of in-links for a webpage chosen uniformly at random is rather different from that
		of the classical Yule model for each finite time $t$ even considering linear rates.
		When $t$ goes to infinity (and for linear rates) the obtained distribution is
		again a Yule distribution but with a different characterizing parameter $\rho$. This is a particularly important
		fact in the sense that considering only the limiting behaviour,
		any empirical Yule distribution recorded on real data can be consequence either
		of an underlying classical Yule model or of a fractional linear Yule model with the same value of $\rho$. Notice however that
		$\rho$ must be interpreted appropriately taking into consideration also the presence of fractionality (given by the value of the
		parameter $\nu$).
		The distribution \eqref{rip} is expressed in a very general form due to the fact that the rates are practically
		unspecified. In Section \ref{satura}, as an example of the different specific cases that can be
		obtained by specializing the rates, we have chosen rates that produce a saturating behaviour.
		Also in this rather realistic specific case, the explicit form of the limiting distribution
		of the number of in-links for a webpage chosen uniformly at random is derived. Figures \ref{clack} and \ref{clack2}
		show that even specializing the rates, the overall shape of the distribution can be quite different.

	\appendix
	\section{Appendix}
	
		\label{appe}	
		For the sake of self-containedness we give here a brief description of the
		fractional nonlinear birth process. For a quick comparison with the classical nonlinear birth process see
		Table \ref{ipad}.
		
		Let us consider a population of individuals
		developing with continuous time $t$ initiated by one single
		initial progenitor at time $t=0$. We indicate the random number of individuals in the population for any fixed time $t$
		with the random variable $\mathfrak{N}^\nu(t)$. It is known \citep{MR2730651} that the state probabilities $p_n^\nu(t) =
		\mathbb{P}(\mathfrak{N}^\nu(t)=n)$, $n \ge 1$, satisfy the system of difference-differential equations
		\begin{equation}
			\label{fra}
			\frac{\textup{d}^\nu p_n^\nu}{\textup{d}t^\nu} = - \lambda_n p_n^\nu + \lambda_{n-1} p^\nu_{n-1}, \qquad n \geq 1,
		\end{equation}
		where $p_0^\nu(s)=0$. Moreover, $p_n^\nu(0) =\delta_{n,1}$, that is the process starts with only one initial progenitor,
		and the fractional derivative is the Caputo derivative (see e.g. \citet{MR2218073,MR2680847}). Briefly, the Caputo derivative is
		an integral operator of convolution-type with a singular power-law kernel. The Caputo derivative can be defined in several
		equivalent ways. We consider here the following form:

		\begin{definition}[Caputo derivative]
			\label{cap}
		    Let $\alpha>0$, $m = \lceil \alpha \rceil$, and $f \in AC^m[a,b]$.
		    The Caputo derivative of order $\alpha>0$ is defined as
		    \begin{equation}
		        \label{Capu}
		        \frac{\textup{d}^\alpha}{\textup{d}t^\alpha} f(t)= \frac{1%
		        }{\Gamma(m-\alpha)}\int_a^{t}(t-s)^{m-1-\alpha}\frac{\textup{d}^m}{\textup{d}s^m}f(s) \, ds.
		    \end{equation}
		\end{definition}		

		In our case we have $\alpha=\nu \in (0,1)$, $m=1$, $a=0$, obtaining
		
		\begin{equation}
			\label{caputo}
			\frac{\textup{d}^\nu}{\textup{d}t^\nu} p_n^\nu(t)= \frac{1}{\Gamma \left( 1- \nu \right)} \int_0^t \frac{ \frac{\textup{d}}{\textup{d}s}
			p_n^\nu \left(
			s \right)}{\left( t-s \right)^\nu} \, \textup{d}s, \qquad 0 < \nu < 1.
		\end{equation}

		It is evident from the above Definition \ref{cap} that the Caputo derivative is a non-local operator in the sense that the integration
		over the interval $(0,t)$ furnishes the system with a persistent memory. Roughly speaking the first order derivative
		$\frac{\textup{d}}{\textup{d}s}p_n^\nu(s)$ is evaluated along the whole time interval $(0,t)$ and weighted by means of the power-law kernel. 
				
		The state probabilities $p_n^\nu (t)$
		of the fractional
		birth process can be explicitly determined and (with the convention that empty products equal unity)
		have the form (for the nonlinear rates $\lambda_j$, $j \ge 1$, all different)
		\citep{MR2730651}
		\begin{equation}
			\label{nlinearnu}
			p_n^\nu (t) =
			\mathbb{P}(\mathfrak{N}^\nu(t)=n)=
			\prod_{j=1}^{n-1} \lambda_j
			\sum_{m=1}^n \frac{ E_{\nu} (-
			\lambda_m t^\nu)}{\prod_{
			l=1,l \neq m}^n \left( \lambda_l -
			\lambda_m \right) }, \qquad n \ge 1, \: t\ge 0,
		\end{equation}
		where $E_{\nu} (\zeta)$ is the so-called Mittag--Leffler function,
		a special function defined as
		\begin{equation}
			\label{ml}
			E_{\nu} \left( \zeta \right) = \sum_{h=0}^\infty
			\frac{\zeta^h}{\Gamma \left( \nu h+1 \right)}, \qquad
			\zeta \in \mathbb{R}, \: \nu > 0,
		\end{equation}
		and having Laplace transform
		\begin{equation}
			\label{mitl}
			\mathcal{L} \bigl( E_\nu(-\lambda t^\nu) \bigr) (z) = \int_0^\infty e^{- z t} E_{\nu} \left( - \xi t^\nu \right) \textup{d}t =
			\frac{z^{\nu -1}}{z^\nu + \xi}, \qquad \nu >0, \: \xi \in \mathbb{R}.
		\end{equation}
		The state probabilities \eqref{nlinearnu} can be actually derived by means of an iterated application of the Laplace transform
		on the equations \eqref{fra} starting from $n=1$. For details on this point see \citet{MR2730651}, Section 2.
		The Mittag--Leffler function \eqref{ml} is in practice a generalization of the exponential function in the
		sense that $E_{1} \left( \zeta \right) = \exp(\zeta)$. General properties of the Mittag--Leffler functions are
		contained in many classical reference books and articles (see e.g.\ the very recent monograph \citet{MR3244285} and the references
		listed therein).

		In the present paper we will often make use of the Laplace transform of the state probabilities \eqref{nlinearnu}
		of the fractional nonlinear birth process.
		From \citet{MR2730651,MR3021490} we we can easily check that
		that
		\begin{align}
			\label{gg}
			\mathbb{L}_n(z) = \int_0^\infty e^{-z t} p_n^\nu(t)\, \textup{d}t
			= z^{\nu-1} \frac{\prod_{r=1}^{n-1}\lambda_r}{\prod_{r=1}^n(z^\nu+\lambda_r)}, \qquad n \ge 1.
		\end{align}
		If the rates are all different, equation \eqref{gg} can be written as
		\begin{align}
			\label{moto}
			\mathbb{L}_n(z) = \prod_{r=1}^{n-1}\lambda_r \sum_{m=1}^n \frac{1}{\prod_{l=1,l\ne m}^n(\lambda_l-\lambda_m)}
			\frac{z^{\nu-1}}{z^\nu+\lambda_m}, \qquad n \ge 1,
		\end{align}
		that is a more manageable form for the purpose of specializing the rates.
		
		For more insights on the properties of the fractional nonlinear birth process see \citet{MR2730651,MR3021490}.
		Here we conclude this section by recalling
		an interesting representation of the fractional nonlinear birth process as a time-changed process and by giving some
		details of the specific case in which the rates are linear. 
		Regarding the first point the fractional nonlinear birth process
		can be constructed as a classical birth process
		stopped at an independent random time given by the inverse process to
		an independent $\nu$-stable subordinator. Notice that stable subordinators are increasing spectrally positive
		L\'evy processes with L\'evy measure given by $m(\textup{d}x)= \left[ \nu/\Gamma(1-\nu) \right]x^{-1-\nu}\textup{d}x$.
		For more details on this last point see \citet{MR1406564, MR3155252}.
		
		An interesting particular case is when the rates are linear, i.e.\ $\lambda_r=\lambda r$. Here the state probability distribution
		\eqref{nlinearnu} specializes in the rather simple form
		\begin{align}
			\label{gengeo}
			p_n^\nu (t) = \sum_{j=1}^n \binom{n-1}{j-1} (-1)^{j-1} E_\nu(-\lambda jt^\nu), \qquad t \ge 0, \: n \ge 1.
		\end{align}
		The geometric distribution of the linear birth process (also known as Yule or Yule--Furry process and indicated in the paper
		with $\mathfrak{N}^\nu_{\text{lin}}(t)$, $t \ge 0$) is retrieved from \eqref{gengeo}
		if the parameter $\nu$ is taken equal to unity. Properties of the fractional Yule process have been studied in
		\citet{MR2479327,MR2730651} while estimators for the intensity $\lambda$ and the fractional parameter $\nu$ have been
		derived in \citet{MR2912344,MR3165549}.
		
		\subsection*{Acknowledgments} P.\ Lansky was supported by the Czech Science Foundation project 15-06991S. L.\ Sacerdote and
			F.\ Polito were supported by the projects ``Application driven Markov and non Markov models'' and ``Stochastic modelling
			beyond diffusions'' (Universit\`a degli Studi di Torino).
		
		\begin{table}\centering
			\scalebox{0.73}{\begin{tabular}{c|c|c|c}
				& \thead{Nonlinear Birth Process} & \thead{Fractional Nonlinear Birth Process} & \thead{Related \\ Formulas} \\ & & & \\
				\hline & & & \\
				\thead{Governing \\ Equations} & $\frac{\textup{d} p_n}{\textup{d}t^\nu}
				= - \lambda_n p_n + \lambda_{n-1} p_{n-1}, \quad n \geq 1$
				& $\frac{\textup{d}^\nu p_n^\nu}{\textup{d}t^\nu} = - \lambda_n p_n^\nu + \lambda_{n-1} p^\nu_{n-1}, \quad n \geq 1$
				& \makecell{\eqref{fra} \\ \vspace{-.3cm} \\ (1.3) of \citep{MR2730651}} \\ & & & \\
				\thead{State \\ Probabilities \\ (rates different)} & $\prod_{j=1}^{n-1} \lambda_j
				\sum_{m=1}^n \frac{ \exp (-
				\lambda_m t)}{\prod_{
				l=1,l \neq m}^n \left( \lambda_l -
				\lambda_m \right) }, \quad n \ge 1$ &
				$\prod_{j=1}^{n-1} \lambda_j
				\sum_{m=1}^n \frac{ E_{\nu} (-
				\lambda_m t^\nu)}{\prod_{
				l=1,l \neq m}^n \left( \lambda_l -
				\lambda_m \right) }, \quad n \ge 1$ & \makecell{\eqref{nlinearnu} \\ \vspace{-.3cm}
				\\ (1.5), (1.11) of \citep{MR2730651}} \\ & & & \\
				\thead{Laplace Transf.\\ of State Probab. \\ (rates different)} &
				$\prod_{r=1}^{n-1}\lambda_r \sum_{m=1}^n \frac{1}{\prod_{l=1,l\ne m}^n(\lambda_l-\lambda_m)}
				\frac{1}{z+\lambda_m}, \quad n \ge 1$ &
				$\prod_{r=1}^{n-1}\lambda_r \sum_{m=1}^n \frac{1}{\prod_{l=1,l\ne m}^n(\lambda_l-\lambda_m)}
				\frac{z^{\nu-1}}{z^\nu+\lambda_m}, \quad n \ge 1$ & \makecell{\eqref{moto} \\ \vspace{-.3cm} \\
				(2.26 of \citep{MR2730651} } \\ & & & \\
				\thead{Mean} & $1 + \sum_{k=1}^\infty \left\{ 1-\sum_{m=1}^k \left( \prod_{\substack{l=1\\l \ne m}}^k
				\frac{\lambda_l}{\lambda_l-\lambda_m} \right) \exp (-\lambda_m t) \right\}$
				& $1 + \sum_{k=1}^\infty \left\{ 1-\sum_{m=1}^k \left( \prod_{\substack{l=1\\l \ne m}}^k
				\frac{\lambda_l}{\lambda_l-\lambda_m} \right) E_\nu (-\lambda_m t^\nu) \right\}$ & \makecell{\eqref{ps2} \\ \vspace{-.3cm} \\ (3.14)
				of \cite{MR3021490}} \\ & & & \\
				\thead{Density of $k$th \\ Waiting Time} & $\lambda_k \exp (-\lambda_k s), \quad k\geq 1$
				& $\lambda_k s^{\nu-1} E_{\nu,\nu} (-\lambda_k s^\nu), \quad k\geq 1$ & (3.8) of \citep{MR3021490}
			\end{tabular}}
			\caption{\label{ipad}Summary of the basic properties of the nonlinear birth process $\mathfrak{N}^1$ (second column) compared with the
				corresponding of the fractional nonlinear birth process $\mathfrak{N}^\nu$ (third column). The last column gives
				references for the corresponding property.}
		\end{table}		
		
	\bibliographystyle{abbrvnat}
	\bibliography{paper}

\end{document}